\documentclass[11pt]{amsart}
\usepackage{graphicx}
\usepackage{epstopdf}
\usepackage{url}
\usepackage{latexsym}
\usepackage{amsfonts,amsmath,amssymb}
\newtheorem{theorem}{Theorem}[section]
\newtheorem{proposition}[theorem]{Proposition}
\newtheorem{lemma}[theorem]{Lemma}
\newtheorem{definition}[theorem]{Definition}

\date{\today}

\begin{document}

\title[Stack words and a bound for 3-stack sortable permutations]{Stack words and a bound
for 3-stack sortable permutations}

\author{Mikl\'os B\'ona}
\address{Department of Mathematics, University of Florida, $358$ Little Hall, PO Box $118105$,
Gainesville, FL, $32611-8105$ (USA)}
\email{bona@ufl.edu}

\begin{abstract} We use stack words to find a new, simple  proof  for the best known upper bound for the number of 3-stack sortable permutations
of a given length. This is the first time that stack words are used to obtain such a result. 
\end{abstract}

\maketitle

\section{Introduction}
Let $p=p_1p_2\cdots p_n$ be a permutation. 
In order to stack sort $p$, we consider
its entries one by one. First, we 
take $p_1$,  and put it in a one-dimensional sorting device,  {\em the  stack}.
The stack stands vertically, and is open at the top; it can hold entries that are increasing from the top 
of the stack to the bottom of the stack.  Second, we  
take $p_2$. If $p_2<p_1$, then
it is allowed for $p_2$ to go in the stack on top of $p_1$, so we  put $p_2$
 there. If $p_2>p_1$, however, then first we take $p_1$ out of the stack,
 and put it
to the first position of the output permutation, and {\em then} we put
$p_2$ into the stack.
We continue this way: at step $i$, we compare $p_i$ with the entry
$r$
 currently on the top of the stack. If $p_i<r$, then $p_i$ goes on the top
of the stack; if not, then $r$ goes to the next empty position of
the output permutation, and $p_i$ gets compared to the new entry that is
now
at the top of the stack. The algorithm ends when all $n$ entries passed
through the stack and are in the output permutation $s(p)$. See Section 8.2 of
\cite{BonCo} for a survey on stack sortable permutations. 
  
\begin{definition} \label{stacksortdef1}
If the output permutation $s(p)$ defined by the above algorithm is the
identity permutation $123\cdots n$, then we say that $p$ is {\em stack
sortable}. 
\end{definition}

It is well-known that a permutation $p$ is stack sortable if and only if it avoids the pattern
231, hence the number of stack sortable permutations of length $n$ is the Catalan number
$C_n={2n \choose n}/(n+1)$.

A permutation is called $t$-stack sortable if sending it through the stack $t$ times results in the identity permutation.
In other words, $p$ is $t$-stack sortable if $s^t(p)$ is the identity permutation. Enumerating $t$-stack sortable permutations
for $t>1$ is extremely difficult.  Let $W_t(n)$ be the number of $t$-stack sortable permutations of length $n$. 
 For $t=2$, the following formula was conjectured by  West in \cite{WES}.  The formula turned out to be quite difficult to prove, but now has 
several complicated proofs. See  \cite{ZEI} , \cite{JAC}, \cite{GOU} or \cite{DEF1} for various proofs. 

\begin{theorem} \label{t:2stack} For all positive integers $n$, the number of 2-stack sortable permutations of length $n$
is given by
\[W_2(n) = \frac{2}{(n+1)(2n+1)} {3n \choose n}.\]
\end{theorem}

There is no formula known for $W_t(n)$ if $t\geq 3$. Even good upper bounds are difficult to find. An easy upper bound
is given by the inequality $W_t(n)\leq (t+1)^{2n} $, which follows from the fact that if $p$ is $t$-stack sortable, then 
$p$ avoids the pattern $23\cdots n1$. Characterizing $t$-stack sortable permutations is very complicated if $t>2$, though
the case of $t=3$ has been done by Claesson and \'Ulfarsson \cite{CLA}, and by  \'Ulfarsson \cite{ULF}. For the case
of general $t$, a less concrete description was recently obtained by Albert, Bouvel and F\'eray in \cite{ABF}. 

Colin Defant \cite{DEF} proved upper bounds for $\lim_{n\rightarrow \infty} \sqrt[n]{W_t(n)}$ for the cases of 
$t=3$ and $t=4$ that are better than the $(t+1)^2$ bound mentioned above. For $t=3$, he proved the upper bound
$\lim_{n\rightarrow \infty} \sqrt[n]{W_3(n)}\leq 12.53296$, and for $t=4$, he proved the upper bound
$\lim_{n\rightarrow \infty} \sqrt[n]{W_3(n)}\leq 21.97225$. Results on related lower bounds can be found in Defant's 
paper \cite{DEF1}, where it is shown that $\lim_{n\rightarrow \infty} \sqrt[n]{W_3(n)} \geq 8.659702$ and that
 $\lim_{n\rightarrow \infty} \sqrt[n]{W_t(n)} \geq (\sqrt{t}+1)^2 $, along with a new proof for the formula for $W_2(n)$,
and a polynomial time algorithm to compute the numbers $W_3(n)$.  

In this paper, we provide a new, simpler proof for Defant's upper bound for the case of $t=3$. Another feauture of our
proof is that it uses stack words, that have been long known to describe $t$-stack sortable permutations, but never 
used to prove an upper bound for them. 

\section{Our method}
 Consider the following modification of the $t$-stack sorting
operation. Instead of passing a permutation through a stack $t$ times,
we pass it 
through $t$  stacks placed next to each other in series
 as follows. The first stack operates
as the usual stack except that when an entry $x$ leaves it, it does not go 
to the output right away. It goes to the next stack if $x<j$, where $j$ is
the entry on the top of the next stack, or if the next stack is empty. If $j<x$, then $x$ cannot move until
$j$ does.

The general step of this algorithm is as follows. Let $S_1,S_2,\cdots
 ,S_t$ be
the $t$ stacks, with $a_i$ being the entry on top of stack $S_i$. If the next
entry $x$ of the input is smaller than $a_1$, we put $x$ on top of 
$S_1$. Otherwise, we find the smallest $i$ so that $a_i$ can move to the next
stack (that is, that $a_{i}<a_{i+1}$ or 
$S_{i+1}$ is empty), and move
$a_i$ on top of $S_{i+1}$. If we do not find such $i$, or if $S_1,S_2,\cdots 
,S_{t-1}$ and the input have all been  emptied out,  then we put the entry on 
the top of $S_t$ into the output.

We can describe the movement of the entries of the input permutation $p$ through the stack by {\em stack words}. 
If $t=1$, then there are just two kinds of moves, an entry either moves in the stack or out of the stack. Let us denote these steps with letters $A$ and $B$, respectively. Then the movement of all entries of $p$ is described by a stack word
consisting of $n$ copies of $A$ and $n$ copies of $B$ in which for all $i$, the $i$th $A$ precedes the $i$th $B$. 
The number of such words is well-known to be $C_n={2n \choose n}/(n+1)$. On the other hand, if $p$ is stack sortable,
then its output is the identity, so given the stack word of $p$, we can uniquely recover $p$. It is easy to prove by
strong induction that each word that satisfies the conditions described in this paragraph is indeed the stack word of 
a stack sortable permutation, so this is a stack word
proof of the fact that $W_1(n)=C_n$. 

In general, if there are $t$ stacks, then there are $t+1$ different kinds of moves. Therefore,  the movement of $p$
through the $t+1$ stacks can be described by a word of length $(t+1)n$ that consists of $n$ copies of each of $t+1$
different letters.  

In particular, if $t=2$, then there are three kind of moves, and if $t=3$, the case that is the subject of our paper, then there are four kinds of moves, which we will treat as follows. 
Let $A$ denote the move of an entry from the input to the first stack, let $B$ denote the move of an entry from the first stack to the second stack, let $C$ denote the move of an entry from the second stack to the third stack, and let $D$ denote the move of an entry from the third stack to the output. 

Note that we will also call stack words 3-stack words or 2-stack words when we want to emphasize the number of stacks that are used to sort a given word. Also note, for future reference, that for all $t$, the descents of $p$ are in bijective
correspondence with the $AA$ factors of $p$. (An $XY$ factor of a word is just a letter $X$ immediately followed by a letter $Y$.)

We will identify 3-stack sortable permutations with their stack words. We can do that since if $p$ is 3-stack sortable, 
then its image under the 3-stack sorting algorithm is the identity permutation, so given the stack word of $p$, 
we can uniquely recover $p$. 

\begin{proposition} \label{ruleset1}
Let $w$ be a 3-stack word of a permutation. Then all of the following hold. 
\begin{enumerate}
\item There is no $BB$ factor in $w$.
\item There is no $CC$ factor in $w$.
\item There is no $BAB$ factor in $w$. 
\item There is no $CBA^jC$ factor in $w$, where $j\geq 0$. 
\end{enumerate}
\end{proposition}

\begin{proof}
Each of these statements holds because otherwise the entries in the second or third stack would not be increasing 
from the top of the stack to the bottom of the stack. 
\end{proof}

\begin{proposition} \label{ruleset2}
Let $w$ be a 3-stack word of a permutation. Then all of the following hold.
\begin{enumerate}
\item There is no $DA$ factor in $w$. 
\item There is no $DB$ factor in $w$.
\item There is no $CA$ factor in $w$.
\end{enumerate}
\end{proposition}

\begin{proof} Each of these statements holds because of the greediness of our algorithm. For instance, a $D$ cannot be followed by an $A$, since the move corresponding to $D$ did not change the content of the first stack, so if the $A$ move
was possible after the $D$ move, it was possible before the $D$ move, and therefore, {\em it would have been made}
before the $D$ move. Analogous considerations imply the other two statements.
\end{proof}

Note that the conditions given in Propositions \ref{ruleset1} and \ref{ruleset2} are necessary, that is, they must hold
in 3-stack words of all permutations, but they are {\em not} sufficient. In other words, if a word satisfies all these conditions, it is not necessarily the 3-stack word of a permutation.

Let $w$ be a 3-stack word of a {\em 3-stack sortable} permutation $p$, and let $v=v(w)$ be the subword of $w$ that consists of the letters $B$, $C$ and $D$ in $w$. In other words, $v=v(w)$ is the word obtained from $w$ by removing all copies
of the letter $A$. This can create $BB$ factors in $v$, even though there were no $BB$ factors in $w$. 

Note that $v$ describes how the stack sorted image $s(p)$ of $p$ traverses the second and third stacks. Note 
that as $p$ is 3-stack sortable, $s(p)$ is 2-stack sortable. So $v$ is the {\em 2-stack word} of the 2-stack sortable permutation
$s(p)$ over the alphabet $\{B,C,D\}$. Therefore, there are $W_2(n)=\frac{2}{(n+1)(2n+1}{3n\choose n}$ 
possible choices for $v$. 

Furthermore,  every descent of $s(p)$ bijectively  corresponds to a $BB$-factor of $v$. The number of 
2-stack sortable permutations of length $n$ with $k-1$ descents is known (see Problem Plus 8.1 in the book \cite{BonCo}) to be
\begin{equation} \label{descentform} W_2(n,k-1)=\frac{(n+k-1)!(2n-k)!}{k!(n+1-k)!(2k-1)!(2n-2k+1)!} .
\end{equation}

\section{Computing the upper bound}
\begin{lemma} \label{sumlemma} The number $W_3(n)$ of 3-stack sortable permutations of length $n$ satisfies the inequality
\[W_3(n) \leq \sum_{k=1}^{(n+1)/2} \frac{(n+k-1)!(2n-k)!}{k!(n+1-k)!(2k-1)!(2n-2k+1)!}  \cdot {2n-2k \choose n-1}.\]
\end{lemma}

\begin{proof}
Let us count all such permutations with respect to the number of descents of their stack sorted image $s(p)$. 
If $s(p)$ has $k-1$ descents,  then its 2-stack word $v$ has $k-1$ factors $BB$. In order to recover the 3-stack word
$w$ of $p$, we must insert $n$ copies of $A$ into $v$ so that we get a valid 3-stack word. As $BB$ factors and $BAB$
factors are not allowed in $w$, we must insert two copies of $A$ into the middle of every $BB$ factor, and we also have to
put one A in front of the first $B$. We have $n-2(k-1)-1=n-2k+1$ copies of $A$ left. We can insert these only in $n$ possible slots, namely on the left of the first $B$, and immediately following any $B$ except the last one. (This is because
Proposition \ref{ruleset2} tells us that there are no $CA$ or $DA$ factors in $w$.) Therefore, by a classic balls-and-boxes
argument, the number of ways to place all copies of $A$ is at most
\[{n-2k+1+n-1\choose n-1 } = {2n-2k\choose n-1}.\] As there are $W_2(n,k-1)$ choices for $v$, the proof is complete
by summing over all possible values of $k$. 
\end{proof}

\begin{theorem} The inequality 
\[ \lim_{n\rightarrow \infty} \sqrt[n]{W_3(n)} \leq 12.53296 \]
holds. 
\end{theorem} 

\begin{proof} As Lemma \ref{sumlemma} provides an upper bound for $W_3(n)$ as a sum of less than $n$ summands, 
it suffices to prove that the largest of those summands is of exponential order 12.539547. In order to do that, 
we use Stirling's formula that states that $m! \sim (m/e)^m \sqrt{2\pi m}$, so $\lim_{m\rightarrow \infty } 
\sqrt[m]{m!}= m/e$. Let $w_3(n,k)$ denote the number of 3-stack sortable permutations $p$ of length $n$ so that
$s(p)$ has $k$ descents.  Setting $k=nx$, with $x\in (0,1]$, and applying Stirling's formula to each factor of the bound in 
Lemma \ref{sumlemma}, this leads to the equality 
 \begin{eqnarray*} g(x): & = & \lim_{n\rightarrow \infty} \sqrt[n]{w_3(n,xn)} \\
& = & 
\frac{(1+x)^{1+x}\cdot (2-x)^{2-x}\cdot (2-2x)^{2-2x}}{ x^x \cdot 
(1-x)^{1-x} \cdot (2x)^{2x}\cdot (2-2x)^{2-2x}\cdot (1-2x)^{1-2x}} \\
& = &  (1+x)\cdot (2-x)^{2-x} \cdot x^{-3x} \cdot (1-x)^{-1+x} \cdot (1-2x)^{2x-1} \cdot \left(\frac{x+1}{4} \right)^x.
 \end{eqnarray*}

The function $g$ takes its maximum when $g'(x)=0$, which occurs when 
\[x=\frac{1}{12} \cdot (27+12\cdot \sqrt{417})^{1/3}-\frac{13}{4\cdot (27+12 \cdot \sqrt{417})^{1/3}}+\frac{1}{4}
\approx 0.2883918927.\]
For that value of $x$, we get $g(x)=12.53296$, completing the proof. 
\end{proof}

\begin{center}  {\bf Acknowledgment}  \end{center}

The author's research is partially supported by Simons Collaboration Grant  421967.


\begin{thebibliography}{AMM8a}
\bibitem{ABF}
Michael Albert, Mathilde Bouvel, and Valentin F\'eray.
Two first-order logics of permutations. 
{\it J. Combin. Theory Ser. A,} {\bf 171} (2020), 46 pages. 

\bibitem{BonCo}
Mikl\'os B\'ona.
{\it Combinatorics of Permutations, second edition}.
CRC Press -- Chapman Hall, Boca Raton, FL, 2012.

\bibitem{CLA}
Anders Claesson and Henning \'Ulfarsson.
Sorting and preimages of pattern classes.
Preprint, available at \url{https://arxiv.org/abs/1203.2437}.


\bibitem{DEF} 
Colin Defant.
Preimages Under the Stack-Sorting Algorithm.
{\it Graphs and Combinatorics} 33:103--122, 2017.

\bibitem{DEF1}
Colin Defant.
{\it Counting 3-Stack-Sortable Permutations} 
Preprint, available at \url{https://arxiv.org/abs/1903.09138}. 

\bibitem{GOU}
Ian P. Goulden, Julian West.
 Raney paths and a combinatorial relationship between rooted nonseparable planar maps and two-stack-sortable permutations.
{\it J. Combin. Theory Ser. A} 75:2, 220--242, 1996.

\bibitem{JAC}
Benjamin Jacquard, Gilles Schaeffer. 
A bijective census of nonseparable planar maps.
{\it J. Combin. Theory Ser. A} 83:1, 1--20, 1998.

\bibitem{ULF} 
Henning \'Ulfarsson.
Describing West-3-stack-sortable permutations with permutation patterns. 
{\it Sém. Lothar. Combin. 67}, Art. B67d, 20 pp., 2011-2012.

\bibitem{WES} 
Julian West.
Permutations with forbidden subsequences and stack-sortable permutations. 
{\it Thesis (Ph.D.)–Massachusetts Institute of Technology. 1990.}

\bibitem{ZEI}
Doron Zeilberger.
A proof of Julian West's conjecture that
 the number of two-stack-sortable permutations of length $n$ is
 $2(3n)!/((n+1)!(2n+1)!)$. 
{\it Discrete Math.} 102 :  85-93, 1992.

\end{thebibliography}
\end{document}